\documentclass[11pt]{amsart}%
\usepackage{amsfonts}
\usepackage{amssymb}
\usepackage{amsmath}
\usepackage{graphicx}%
\providecommand{\U}[1]{\protect\rule{.1in}{.1in}}
\newtheorem{proposition}{Proposition}[section]
\newtheorem{theorem}{Theorem}[section]

\newtheorem{lemma}{Lemma}[section]

\newtheorem{remark}{Remark}[section]

\numberwithin{equation}{section}

\newcommand{\abs}[1]{\lvert#1\rvert}
\newcommand{\A}{\mathbb{A}}
\newcommand{\R}{\mathbb{R}}
\newcommand{\C}{\mathbb{C}}
\newcommand{\Z}{\mathbb{Z}}

\newcommand{\Q}{\mathbb{Q}}

\newcommand{\cir}{\mathbb{S}^1}
\newcommand{\distr}{\mathcal{D}'}
\newcommand{\ov}{\overline}
\newcommand{\ep}{\epsilon}
\newcommand{\dis}{\displaystyle}
\newcommand{\pa}{\partial}
\newcommand{\ei}[1]{\textrm{e}^{#1}}
\newcommand{\dd}[2]{\dis\frac{\pa #1}{\pa #2}}
\newcommand{\ta}{\theta}
\newcommand{\px}{\partial_x}
\newcommand{\py}{\partial_y}
\newcommand{\pr}{\partial_r}
\newcommand{\pt}{\partial_\theta}
\newcommand{\condp}{\mathrm{condition}\,(\mathcal{P})}
\newcommand{\re}[1]{\mathrm{Re}(#1)}
\newcommand{\Si}{\Sigma}
\newcommand{\Om}{\Omega}

\newcommand{\lam}{\lambda}

\newcommand{\baks}{\backslash}
\newcommand{\mcald}{\mathcal{D}'}
\newcommand{\mcaldc}{(\mathcal{DC})}
\newcommand{\mcaldcc}{(\mathcal{DC}')}

\newcommand{\ccinf}{C^\infty_{c}}
\newcommand{\cinf}{C^\infty}
\def\supp{\hbox{\rm supp\,}}
\def\<{\langle}
\def\>{\rangle}

\newcommand{\begeq}{\begin{equation}}
\newcommand{\stopeq}{\end{equation}}

\newcommand{\begar}{\begin{array}}
\newcommand{\stopar}{\end{array}}

\begin{document}
\setcounter{page}{1}

\author{Abdelhamid Meziani}
\address{\small Department of Mathematics, Florida International
University, Miami, FL, 33199, USA e-mail: meziani@fiu.edu}
\title[ Homogeneous singularities ]{Solvability of planar
complex vector fields with  homogeneous singularities}


\begin{abstract} In this paper we study the equation $Lu=f$, where $L$ is a $\C$-valued vector field in $\R^2$ with
a homogeneous singularity. The properties of the solutions are linked to the number theoretic properties of a pair
of complex numbers attached to the vector field.
\end{abstract}

\maketitle

\section*{Introduction}
This paper deals with the solvability of homogeneous complex vector fields in the plane. It is motivated
by the recent papers {\cite{Tre1}}, {\cite{Tre2}}, and {\cite{Tre3}} by Fran\c{c}ois Treves about vector
fields with linear sigularities.
Let $L=A\px+B\py$ be a complex vector field in $\R^2$, where $A,\, B\,\in \cinf(\R^2\baks\{0\})$
are homogeneous with degree $\lam\in\C$ with $\mathrm{Re}(\lam)>1$.
The main question addressed here is the solvability, in the sense of distributions, of the equation
\begeq Lu=f\stopeq
in a domain $\Om$ containing the singular point 0.
The question is answered for a class of vector fields. This class consists of those vector fields
$L$ that are linearly independent of the radial vector field ($L\wedge (x\px+y\py)\ne 0$)
  and such that the characteristic set
$\Si$ has an empty interior in  $\R^2\baks\{0\}$ (see section 1).

An important role is played by the number theoretic properties of the pair of complex numbers
$(\mu,\lam)$ attached to $L$. The second number, $\lam$, is the homogeneity degree of $L$ and the
first is given by
\begeq
\mu = \frac{1}{2\pi}\int_{\abs{z}=1}\!\!
\frac{A(x,y)x+B(x,y)y}{B(x,y)x-A(x,y)y}\,\frac{dz}{z}\, ,
\stopeq
with $z=x+iy$. For instance, if $\mu =i\beta$ (\ $\beta\in\R^\ast$), then whenever $u\in C^0(\Om)$
solves $Lu=0$ in a region containing an annulus $r^2\le x^2+y^2\le R^2$ with $R>r\ei{2\pi\abs{\beta}}$,
the function $u$ extends as a distribution solution to the whole plane $\R^2$.
The solvability of (0.1), when $f$ is real analytic at $0$, can be achieved provided that the pair
of numbers $(\mu,\lam)$ is nonresonant and satisfies a certain Diophantine condition $\mcaldc$
which prevents the formation of small denominators (see section 6).
When $f$ is only $\cinf$ at 0, the Diophantine condition together with $L$ satisfying the Nirenber-Treves
$\condp$ in $\R^2\baks\{0\}$ imply the solvability of (0.1) in a neighborhood of 0.
These results are used to solve a Riemann-Hilbert boundary value problem for $L$.

The paper is organized as follows. Section 1 deals with the preliminaries. In section 2, we
construct first integrals for $L$. The properties of the first integrals depend on the number $\mu$.
Section 3 deals with the equation $Lu=0$ in a region $\Om$ containing the singular point 0.
We consider in particular the extendability of $u$ across the boundary $\pa\Om$ and the smoothness of $u$
on certain characteristic rays of $\Si$. Examples in section 4 illustrate the properties discussed
in section 3. In section 5,
we consider equation (0.1) when $f$ is a homogeneous function with degree $\sigma\in\C$,
with $\mathrm{Re}(\sigma)>0$. We show that (0.1) has a homogeneous distribution $u\in\mcald(\R^2)$
provided that the three numbers $\mu,\ \lam,$ and $\sigma$ satisfy
$\mu(\lam-\sigma-1)\not\in\Z$.
In section 6, we introduce a Diophantine condition $\mcaldc$ for a nonresonant pair $(\mu,\lam)$
and prove that when the pair of numbers satisfies the Diophantine condition,
equation (0.1) has a distribution solution in a neighborhood of $0$ for every
real analytic function $f$.
In section 7, we consider the case when $f$ is only $\cinf$ at 0 but we add the assumption that
$L$ satisfies $\condp$ in $\R^2\baks\{0\}$  and $(\mu,\lam)$ satisfies
$\mcaldc$. In this case equation (0.1) has a distribution solution.
The result of section 7 is used in section 8 to the boundary value problem
\begeq\left\{\begar{ll}
Lu =f & \quad\mathrm{in}\ \Om\, ,\\
\mathrm{Re}(\Lambda u)=\Phi &\quad\mathrm{on}\ \pa\Om\, ,
\stopar\right.\stopeq
 where $\Om\ni 0$ is a simply connected domain, $\Lambda\in C^\sigma(\pa\Om,\cir)$,
 $\Phi\in C^\sigma(\pa\Om,\R)$ with $0<\sigma <\, 1$. We prove that problem (0.3)
 has a distribution solution provided that the index $\kappa$ of $\Lambda$ satisfies
 the inequality
 \[
 \kappa \,\ge\, -1-\frac{\mathrm{Re}(\lam)-1}{\mathrm{Re}(1/\mu)} \, .
 \]

\section{Preliminaries}
In this section we provide the necessary terminology and background regarding the vector fields we will be dealing
with. Let
\begeq
L=A(x,y)\px +B(x,y)\py
\stopeq
be a vector field in $\R^2$, where $A,B\in \cinf (\R^2\baks \{0\}\, ,\C)$ are homogeneous functions of
degree $\lam\in \C$ with $\re{\lam} >1$. This means that, for every $(x,y)\in\R^2$ and for every $t\in\R^+$,
\[
A(tx,ty)=t^\lam A(x,y)\quad\mathrm{and}\quad B(tx,ty)=t^\lam B(x,y)\, .
\]
It follows in particular (when $\re{\lam}>1$) that $A$ and $B$ are at least of class $C^1$ at
$0$ and that
\[
\dd{^{j+k}A}{x^j\pa y^k}(0,0)=\dd{^{j+k}B}{x^j\pa y^k}(0,0)=0\quad\mathrm{when}\ j+k\le [\re{\lam}]\, ,
\]
where $[ a]$,  $a\in\R$, denotes the greatest integer $\le\, a$. The conjugate of $L$ is the vector field
\[
\ov{L}=\ov{A(x,y)}\,\px +\ov{B(x,y)}\,\py\, ,
\]
where $\ov{A}$ and $\ov{B}$ are the complex conjugates of $A$ and $B$. The polar coordinates are the natural
variables for the study of such operators and we will use them extensively. Consider the polar coordinates
map
\[
\Pi :\, [0,\ \infty)\times \cir \, \longrightarrow \, \R^2\baks\{0\}\, ,\ \
\Pi (r,\ta ) = (r\cos\ta ,r\sin\ta)\, .
\]
It follows from the $\lam$-homogeneity of $A$ and $B$ that
\[
A\circ\Pi (r,\ta)=r^\lam a(\ta)\quad B\circ\Pi (r,\ta)=r^\lam b(\ta)
\]
where $a(\ta)=A(\cos\ta,\sin\ta)$ and $b(\ta)=B(\cos\ta,\sin\ta)$ are $2\pi$-periodic and
can be considered as functions in $\cinf(\cir,\C)$.
The expression of $L$ in polar coordinates is
\begeq
L=r^{\lam-1}\left(p(\ta)\pt-iq(\ta)r\pr\right)\, ,
\stopeq
with
\begeq
p(\ta)=b(\ta)\cos\ta -a(\ta)\sin\ta \ \ \mathrm{and}\ \
q(\ta)=i\left(a(\ta)\cos\ta +b(\ta)\sin\ta\right)\, .
\stopeq
The vector field $L$ is elliptic at a point $p$ if $L_p\wedge\ov{L}_p \ne 0$.
The set of points where $L$ is nonelliptic, given by
\begeq
\Si =\left\{ (x,y)\in\R^2;\ L\wedge\ov{L}=0
\right\}=\left\{ (x,y)\in\R^2;\ \mathrm{Im}(A\ov{B})=0\right\}\, ,
\stopeq
is the base projection of the characteristic set of $L$.
It follows from the homogeneity of $A$ and $B$ that $\Si$ is a union of lines through
the origin.

Throughout this paper, we will assume that $L$ satisfies the following conditions
\begin{itemize}
\item[] $(\mathcal{C}1)$: $\ \Si$ has an empty interior in $\R^2$;
\item[] $(\mathcal{C}2)$: $\ L\wedge \left(x\px+y\py\right)\ne 0$ everywhere in $\R^2\baks\{0\}$.
\end{itemize}
Note that since
\begeq
\mathrm{Im}(A\ov{B})=r^{2\re\lam}\mathrm{Im}(a(\ta)\ov{b(\ta)})=r^{2\re\lam}\re{p(\ta)\ov{q(\ta)}}\, ,
\stopeq
then when $\Si$ is viewed as a subset in $[0,\ \infty)\times \cir$, we have
\begeq
\Si =[0,\ \infty)\times\Si_0\,\ \ \mathrm{with}\ \ \Si_0=\{\ta\in\cir;\ \re{p\ov{q}}=0\}\, .
\stopeq
The following lemma is a direct consequence of (1.2), (1.6) and of $x\px+y\py=r\pr$.
\begin{lemma}
Condition $(\mathcal{C}1)$ is equivalent to $(\mathcal{C}1')$ and to $(\mathcal{C}1'')$.
\begin{itemize}
\item[] $(\mathcal{C}1')$ The set $\{ \ta\in\cir\, ;\ \mathrm{Im}(a\ov{b})=0\}$ has an
empty interior in $\cir$.
\item[] $(\mathcal{C}1'')$ The set $\{ \ta\in\cir\, ;\ \re{q\ov{p}}=0\}$ has an
empty interior in $\cir$.
\end{itemize}
Condition $(\mathcal{C}2)$ is equivalent to $(\mathcal{C}2\,')$ and to $(\mathcal{C}2\,'')$.
\begin{itemize}
\item[] $(\mathcal{C}2\,')$ For every $\ta\in\cir$, $\ b(\ta)\cos\ta -a(\ta)\sin\ta \ne 0$.
\item[] $(\mathcal{C}2\,'')$ For every $\ta\in\cir$, $p(\ta)\ne 0$.
\end{itemize}
\end{lemma}

The vector field $L$ is said to be locally solvable at a point $p\in\R^2$ if there exist open sets
$U, V$ with $p\in V\subset U\subset\R^2$ such that for every $f\in\cinf(U)$ there exists a distribution
$u\in\mcald (V)$ such that $Lu=f$. Note that since $L$ is elliptic in $\R^2\baks\Si$, then it is
locally solvable and hypoelliptic at each point $p\notin\Si$. The vector field is solvable at a point
$p_0\in\Si\baks\{0\}$ if and only if $L$ satisfies the Nirenberg-Treves $\condp$. This condition deals with
general differential operators (see {\cite{Tre-Boo}}, {\cite{Nir-Tre}}). In the case of vector fields we
consider here, it has a simple formulation. Namely:
\emph{ The homogeneous vector field $L$ satisfies $\condp$ at the point
$(x_0,y_0)=(r_0,\ta_0)\in\Si\baks\{0\}$  if and only if the function $\mathrm{Im}(A\ov{B})$ does not change
sign in a neighborhood of $(x_0,y_0)$. Equivalently, if and only if $\mathrm{Im}(a\ov{b})=\re{q\ov{p}}$ does not
change sign in a neighborhood of $\ta_0\in\cir$.}
It follows from the homogeneity of $L$ that if it satisfies $\condp$ at a point $p_0$, then it satisfies
$\condp$ on the ray $\{ tp_0\in\C,\ t>0\}$.
It should be noted that when $L$ satisfies $\condp$ at a point $p_0\in\Si\baks\{0\}$, then for $f\in\cinf$
near $p_0$, a $\cinf$ function $u$ can be found so that $Lu=f$ near $p_0$.

\begin{remark}
{\rm{Condition $(\mathcal{C}2)$ implies that $L$ is not tangent to any ray contained in $\Si$.
In terms of (Sussmann's) orbits
of $L$ (see {\cite{Ber-Cor-Hou}} and {\cite{Tre-Boo}}), this condition implies that $L$ does not have one-dimensional
orbits in $\R^2\baks\{0\}$. As a consequence, if $L$ satisfies conditions $(\mathcal{C}1)$, $(\mathcal{C}2)$ and
$\condp$ at a point $p_0\in\R^2\baks\{0\}$,
then it is hypoelliptic in a neighborhood of $p_0$ (or equivalently, with the terminology
of hypoanalytic structures ({\cite{Ber-Cor-Hou}}, {\cite{Tre-Boo}}) it is hypocomplex in a neighborhood of $p_0$).
In this case  all solutions of $Lu=f$ are $\cinf$ at $p_0$ (respectively, real-analytic if $L$ is real
analytic in $\R^2\baks\{0\}$.)}}
\end{remark}

\begin{remark}{\rm
If $p_0=(r_0,\ta_0)\in\Si\baks\{0\}$ and if the function $\re{q\ov{p}}$ vanishes to a finite order at $\ta_0$,
then $L$ is of finite type at $p_0$. This means that the Lie algebra generated by $L$ and $\ov{L}$ generates the
complexified tangent space $\C T_{p_0}\R^2$. If, in addition, $L$ satisfies $\condp$ at $p_0$, then the
order of vanishing of $\re{q\ov{p}}$ is even ($\re{q\ov{p}}(\ta)=O((\ta-\ta_0)^{2k})$ for some $k\in\Z^+$ as
$\ta\, \longrightarrow\, \ta_0$). Otherwise, $L$ is of infinite type at $p_0$. }
\end{remark}

\begin{remark}
{\rm When $L$ is real analytic in $\R^2$, then the functions $A$ and $B$ are homogeneous polynomials
of degree $N+1\ge 2$ so that
\[
A(x,y)=\sum_{j=0}^{N+1}A_jx^{N+1-j}y^j\quad\mathrm{and}\quad B(x,y)=\sum_{j=0}^{N+1}B_jx^{N+1-j}y^j
\]
with $A_j\, ,\, B_j\,\in\C$. Then, we can express $L$ in polar coordinates as
\begeq
L=r^N\left(p_{N+2}(\ta)\pt-iq_{N+2}(\ta)r\pr\right)\, ,
\stopeq
where $p_{N+2}$ and $q_{N+2}$ are trigonometric polynomials of degrees $\le N+2$, i.e. polynomials
of the form
\[
\sum_{j=0}^{N+1}c_j\ei{i(N+2-2j)\ta}
\]
with $c_j\in\C$. It follows, in particular, that the characteristic set $\Si$ consists of at most
$2N+2$ lines through the origin.}
\end{remark}

\section{First integrals}
We give here the properties of the first integrals of $L$. By a first integral in a domain $\Om$,
we mean a function $F$ such that $LF=0$ and $dF$ is nowhere vanishing in $\Om$.
Throughout this paper we will assume that $L$ is $\cinf$  in $R^2\baks\{0\}$,
that it is $\lam$-homogeneous, and it satisfies
conditions $(\mathcal{C}1)$ and $(\mathcal{C}2)$. With such a vector field,
given in polar coordinates by (1.2),  we associate the complex number
\begeq
\mu =\mu(L)=\frac{1}{2\pi}\int_0^{2\pi}\frac{q(\ta)}{p(\ta)}d\ta\, .
\stopeq
The number $\mu$ can also be expressed as
\[
\mu =\frac{1}{2\pi}\int_0^{2\pi}\frac{a(\ta)\cos\ta+b(\ta)\sin\ta}{b(\ta)\cos\ta-a(\ta)\sin\ta}d\ta
=\frac{1}{2\pi}\int_{|z|=1}\frac{A(x,y)x+B(x,y)y}{B(x,y)x-A(x,y)y}\,\frac{dz}{z}
\]
where $z=x+iy$.
The number $\mu$ is introduced in {\cite{Mez-JFA}} (see also {\cite{Cor-Gon}}) for a class of elliptic
vector fields that degenerate along a closed curve.
Note that if $\mathrm{Re}({\mu}(L))\le 0$, then a linear change of variables in $\R^2$
(for example $\widetilde{x}=-x,\ \widetilde{y}=y$) transforms $L$ into a new, homogeneous vector
field $\widetilde{L}$ with $\mathrm{Re}({\mu}(\widetilde{L}))\ge 0$. Hence, from now on we will assume
that
\begeq
\re{\mu}\ge 0\, .
\stopeq

\begin{lemma}
If $L$ satisfies $\condp$ in $\R^2\baks\{0\}$, then $\re{\mu} >0$.
\end{lemma}

\begin{proof}
We have
\[
\re{\mu} =\frac{1}{2\pi}\int_0^{2\pi}\re{\frac{q(\ta)}{p(\ta)}}d\ta =
\frac{1}{2}\int_0^{2\pi}\frac{\re{q(\ta)\ov{p(\ta)}}}{|p(\ta)|^2}d\ta
\]
Since $\re{q\ov{p}}\not\equiv 0$ (condition $(\mathcal{C}1)$) and does not change sign ($\condp$),
then the last integral is nonzero. The conclusion follows from (2.2)
\end{proof}

Consider the case $\mu =0$. Then, the function
\begeq
m(\ta)=\int_0^\ta \frac{q(s)}{p(s)}\,ds
\stopeq
is $2\pi$-periodic and $\re{m(\ta)}\not\equiv 0$ (condition $(\mathcal{C}1)$).
Let $\ta_1,\,\ta_2\in [0,\ 2\pi)$ and $\sigma >0$ be such that
\begeq\begar{c}
\dis\re{m(\ta_1)}=\min_{0\le\ta\le 2\pi}\re{m(\ta)}\, ,\quad
\re{m(\ta_2)}=\max_{0\le\ta\le 2\pi}\re{m(\ta)}\\
 \dis \sigma=\frac{\pi}{ \re{m(\ta_2)}-\re{m(\ta_1)}}\, .
\stopar\stopeq
Let
\begeq
\phi(\ta)=\sigma m(\ta)-\sigma\re{m(\ta_1)}\, .
\stopeq
We have then
\begeq
\min_{0\le\ta\le 2\pi}\re{\phi(\ta)}=0\quad\mathrm{and}\quad
\max_{0\le\ta\le 2\pi}\re{\phi(\ta)}=\pi\, .
\stopeq
Define the homogeneous function in $\R^2$ by
\begeq
Z_0(r,\ta)=r^\sigma \ei{i\phi(\ta)}\, .
\stopeq

\begin{proposition}
Suppose that $\mu =0$. Let $Z_0$ be the function defined by $(2.7)$. Then
 $\ Z_0\in\cinf(\R^2\baks\{0\})$
is a first integral of $L$ in $\R^2\baks\{0\}$ and
it maps $\R^2$ onto the upper half plane $\C^+=\{z=x+iy\in\C ;\ y\ge 0\}$.
\end{proposition}

\begin{proof} The fact that
 $Z_0\in\cinf(\R^2\baks\{0\})$ and solves $LZ_0=0$ follows from the construction of
the function $\phi(\ta)$ which satisfies $\phi'(\ta)=\sigma\dis\frac{q(\ta)}{p(\ta)}$.
That $dZ_0\ne 0$ in $\R^2\baks\{0\}$ is also clear. We need to verify that
$Z_0(\R^2)=\C^+$. Let $\phi_1$ and $\phi_2$ be, respectively, the real and imaginary parts of $\phi$.
The function $\ei{i\phi}=\ei{-\phi_2}\ei{i\phi_1}$ maps the unit circle in $\R^2$ onto
a curve contained in $\C^+$ (since $0\le \phi_1(\ta)\le\pi$ for every $\ta$, by definition of $\phi$).
If the maximum $\pi$ and minimum 0 of $\phi_1$ occurs at $\ta_2$ and $\ta_1$, then $\ei{i\phi(\ta_2)}\in\R^-$
and $\ei{i\phi(\ta_1)}\in\R^+$. The $\sigma$-homogeneity of $Z_0$ implies $Z_0(R^2)=C^+$.
 \end{proof}

 When $\mu\ne 0$, then, as it is the zeroth Fourier coefficient of the function $\dis\frac{q}{p}$,
 we can write
 \begeq
 \frac{q(\ta)}{p(\ta)}=\mu\left(1+\sum_{j\in\Z,\, j\ne 0}\frac{c_j}{\mu}\ei{ij\ta}\right)
 \stopeq
 with $\dis c_j=\frac{1}{2\pi}\int_0^{2\pi}\frac{q(\ta)}{p(\ta)}\ei{-ij\ta}d\ta$.
The sum appearing in (2.8) has a periodic primitive
\begeq
\phi(\ta)=\sum_{j\in\Z,\, j\ne 0}\frac{c_j}{ij\mu}\ei{ij\ta}=\phi_1(\ta)+i\phi_2(\ta)
\stopeq
with $\phi_1$ and $\phi_2$ the real and imaginary parts of $\phi$. Hence,
\begeq
\frac{q(\ta)}{p(\ta)}=\mu (1+\phi'(\ta))\, .
\stopeq
Let $Z_\mu$ be the function defined by
\begeq
Z_\mu (r,\ta)=r^{1/\mu}\exp\left(i(\ta+\phi(\ta))\right)\, .
\stopeq
The function $Z_\mu$ is $\cinf$ in $\R^2\baks\{0\}$. If $\re{\mu}>0$, then, at $0$,
$Z_\mu$ is of class $C^k$  with $k<\re{1/\mu}$
(hence, only H\"{o}lder continuous, if $0\le \re{1/\mu} <1$). If $\re{\mu}=0$, $Z_\mu$ is not defined
at $0$, but it is  bounded in the whole punctured plane $\R^2\baks\{0\}$.
Furthermore, it is easily verified that $LZ_\mu =0$ and $dZ_\mu \ne 0$ in $\R^2\baks\{0\}$.
That is, $Z_\mu$ is a first integral of $L$ in $\R^2\baks\{0\}$.

\begin{proposition}
Suppose that $\re{\mu} >0$. Then, the first integral $Z_\mu$ maps $\R^2$ onto $\C$. Furthermore, if
$L$ satisfies $\condp$ in $\R^2\baks\{0\}$, then $Z_\mu$ is a global homeomorphism.
\end{proposition}

\begin{proof}
Since $\phi(\ta)$ is $2\pi$-periodic, then the map $\exp\left(i(\ta+\phi(\ta))\right)$ sends
the unit circle in $\R^2$ into a curve in $\C$ with winding number 1 about the origin.
The $(1/\mu)$-homogeneity of $Z_\mu$ and $\re{\mu}>0$ imply that $Z_\mu(\R^2)=\C$.
If, in addition, $L$ satisfies $\condp$, then  $\re{1+\phi'(\ta)}=1+\phi_1'(\ta)\ge 0$, for every
$\ta$. It follows from condition $(\mathcal{C}1)$, that the function
$\ta+\phi_1(\ta)$ is monotone (increasing). Consequently, $\exp\left(i(\ta+\phi(\ta))\right)$
is a homeomorphism from $\cir$ onto its image. The conclusion follows from the homogeneity of
$Z_\mu$.
\end{proof}

When $\re{\mu}=0$, set $\mu =i\beta$ with $\beta\in\R^\ast$. Then (2.11) can be rewritten as
\begeq
Z_{i\beta}=\ei{-\phi_2(\ta)}\exp\left[i\left(\ta+\phi_1(\ta)-\frac{\ln r}{\beta}\right)\right]\, .
\stopeq
Let
\begeq
m=\min_{0\le\ta\le 2\pi}\ei{-\phi_2(\ta)}\quad\mathrm{and}\quad
M=\max_{0\le\ta\le 2\pi}\ei{-\phi_2(\ta)}\, .
\stopeq
As before, $Z_{i\beta}$ is a first integral of $L$ in $\R^2\baks\{0\}$ and we have the following

\begin{proposition}
The function $Z_{i\beta}$ maps $\R^2\baks\{0\}$ onto the annulus
\begeq
\A (m,M)=\{ z\in\C;\ m\le\abs{z}\le M\}\, .
\stopeq
Furthermore, for every $\ep >0$, the function $Z_{i\beta}$ maps the punctured disc
$D(0,\ep)\baks \{0\}$ onto the annulus $\A(m,M)$.
\end{proposition}

\begin{proof}
Let $(r,\ta)\in \R^2$ ($r>0$). Then $\abs{Z_{i\beta}(r,\ta)}=\ei{-\phi_2(\ta)}$. Hence,
$m\le\abs{Z_{i\beta}}\le M$, and $Z_{i\beta}(r,\ta)\in\A(m,M)$. Now, let
$z_0=\rho_0\ei{i\phi_0}\in\A(m,M)$. Then there exists $\ta_0\in [0,\ 2\pi)$ such that
$\rho_0=\ei{-\phi_2(\ta_0)}$. For a given $\ep >0$, let
\[
r_\ep =\exp\left[\beta\left(\ta_0+\phi_1(\ta_0)-\phi_0+2k\pi\right)\right]
\]
with $k\in\Z$, $k\beta <0$ and $\abs{k}$ large enough so that $r_\ep <\ep$. Then
$Z_{i\beta}(r_\ep,\ta_0)=z_0$.
\end{proof}

\begin{remark}
{\rm When $\mathrm{div}(L)=0$, the differential form $Ady-Bdx$ (orthogonal to $L$)
is exact. Hence, there exists a $(\lam +1)$-homogeneous function $F(x,y)$  such that
\begeq
L=\dd{F}{y}\,\dd{}{x} -\dd{F}{x}\,\dd{}{y}\, .
\stopeq
If we set $F(x,y)=r^{\lam+1}f(\ta)$, then
\begeq
L=r^{\lam-1}\left[(\lam+1)f(\ta)\dd{}{\ta}-f'(\ta)r\dd{}{r}\right]\, .
\stopeq
Let $\arg f(\ta)$ be a continuous branch of the argument of $f$.
The number $\mu$ can be given by
\begeq
\mu =\frac{\arg f(2\pi)-\arg f(0)}{2\pi (\lam+1)}=\frac{j}{\lam+1}
\stopeq
for some $j\in\Z$. If, in addition, $L$ is real analytic at $0$, then $F=P_{N+1}$
is a homogeneous polynomial of degree $N+1\in Z^+$. So $F(r,\ta)=r^{N+1}P(\ta)$,
with $P(\ta)=\dis\sum_{k=-(N+1)}^{N+1}p_k\ei{ik\ta}$ a trigonometric polynomial of degree $N+1$.
Let $R(z)=\dis\sum_{k=-(N+1)}^{N+1}p_kz^k$. Then the associated number is
\[
\mu =\frac{1}{2\pi i(N+1)}\int_{\abs{z}=1}\frac{R'(z)}{R(z)}\, dz =\frac{j}{N+1}
\]
with $j=Q-(N+1)$, where $Q$ is the number of zeros of $R$ in the disc $\abs{z}<1$.
 }
\end{remark}

\begin{remark}
{\rm The functions $Z_\mu$ defined in (2.7), (2.11), and (2.12) satisfy $LZ_\mu =0$ in the
sense of distribution in the whole plane $\R^2$.}
\end{remark}

\section{The equation $Lu=0$}
In this sections we study the properties of the solutions of the
 equation
 \begeq
 Lu=0\,.
 \stopeq
 The following propositions (in which the equations are to be understood in the sense
 of distributions) are direct consequences of the order of
 vanishing of $L$ at $0$.

 \begin{proposition}
 Let $\delta$ be the Dirac distribution and $j,k$ be nonnegative integers such that $j+k\le \re{\lam}-1$.
 Then
 \begeq
 L\left(\dd{^{j+k}\delta}{x^j\pa y^k}\right)=0\, .
 \stopeq
 \end{proposition}

 \begin{proposition}Suppose that $\mu =0$. Let $Z_0$ be the first integral of $L$ given by $(2.7)$
 and let $m\in\Z$ such that $\, \sigma m<\re{\lam}-1$. Then
 \begeq
 L\left(Z_0^{-m}\right)=0\, .
 \stopeq
Suppose that $\re{\mu}>0$. Let $Z_\mu$ be the first integral of $L$ defined by $(2.11)$
and let $m\in\Z$ such that $\,\re{1/\mu} m <\re{\lam}-1$, then
\begeq
L\left(Z_\mu^{-m}\right) =0\, .
\stopeq
 Suppose that $\mu =i\beta\,\in\, i\R^\ast$. Let $Z_{i\beta}$ be the first integral of $L$
 given by $(2.12)$ and $m\in\Z$. Then
 \begeq
 L\left(Z_{i\beta}^m\right)=0\, .
 \stopeq
  \end{proposition}

Now we consider the continuous solutions of (3.1).

\begin{theorem}
Suppose that $L$ satisfies $\condp$ in $\R^2\baks\{0\}$. Let $Z_\mu$ be the first
integral given by $(2.11)$ and let $\Om\subset\R^2$
be open. Then $u\in C^0(\Om)$ satisfies $(3.1)$ if and only if there exists a
function $H$ holomorphic in $Z_\mu(\Om)$ such that $u=H\circ Z_\mu$.
\end{theorem}

This theorem is a direct consequence of the hypocomplexity of $L$ in $\R^2\baks\{0\}$ and
of the fact that $Z_\mu :\, \R^2\,\longrightarrow\, \C$ is a homeomorphism (Proposition 2.2).

\begin{theorem}
Suppose that $\mu=0$ and let $Z_0$ be the first integral of $L$ given by $(2.7)$.
Let $\Om$ be open and simply connected in $\R^2$ and such that $Z_0(\Om)$ is simply connected
in $\C$. Then every solution $u\in C^0(\Om)$  of $(3.1)$ satisfies
$u\in\cinf\left(\Om\baks Z_0^{-1}(\R)\right)$
\end{theorem}

\begin{proof}
We know from the principle of constancy on the fibers of locally integrable structures
(which is a consequence of the Baouendi-Treves approximation formula,
see {\cite{Ber-Cor-Hou}} or {\cite{Tre-Boo}}) that, for every $p\in\Om$, there exists
an open set $ O_p\subset\Om$, with $p\in O_p$, and a function $H_p\in C^0(Z_0(O_p))$, with
$H_p$ holomorphic in the interior of $Z_0(O_p)$, such that $u=H_p\circ Z_0$ in $O_p$.
By analytic continuation, each element $H_p$ leads to (an a priori multi-valued) function
$H$ continuous in $Z_0(\Om)\baks\{0\}$ and holomorphic in its interior.
Now, since $Z(\Om)$ is simply connect in the upper half plane $\C^+$ and $Z_0(0)=0$, then
$H$ is single-valued and the conclusion follows.
\end{proof}

\begin{theorem}
Suppose that $\re{\mu}>0 $. Let $Z_\mu$ be the first integral of $L$ given by $(2.11)$.
Let $\Om\ni 0$ be an open and simply connected bounded subset of
$\R^2$  and such that $Z_\mu(\Om)$ is simply
connected in $\C$. Then  for every $u\in C^0(\Om)$ satisfying  $(3.1)$, we have
$u\in\cinf(\Om\baks\{0\})$.
\end{theorem}

\begin{proof}
As in the proof of Theorem 3.2, for every $p\in\Om$, there exists an open subset $O_p\ni p$
and $H_p$ continuous on $Z_\mu(O_p)$ and holomorphic in its interior, such that $u=H_p\circ Z_\mu$
in $O_p$. The analytic continuation leads to a (multi-valued) function $H$ in $Z_\mu(\Om)\baks\{0\}$.
We need to verify that $H$ is single-valued. This time, $0$ is in the interior of $Z_\mu(\Om)$
and $Z_\mu(\Om)\baks\{0\}$ is not simply connected. Since $\Om$ is simply connected, then $\pa\Om$ is
homotopic in $\Om\baks\{0\}$ to a circle $C(0,\ep)$ with center $0$ and radius $\ep >0$ (small).
The image of this circle by $Z_\mu$ is the curve $\Gamma$ given by $\ep^{1/\mu}\ei{i(\ta+\phi(\ta))}$.
It has winding number 1 about $0\in\C$. Let $p_0\in C(0,\ep)$ be such that $Z_\mu(p_0)\in \Gamma$. If
$H$ were not single-valued, then the analytic continuation of an element $H_1$ of $H$ at $Z_\mu(p_0)$,
with value $H_1(Z_\mu(p_0))$, would lead to an element $H_2$ of $H$, with value
$H_2(Z_\mu(p_0))\ne H_1(Z_\mu(p_0))$. But the continuation along the curve $\Gamma$ is realized
through the continuation along the circle $C(0,\ep)$ of the function $u$ leading to
$H_1(Z_\mu(p_0))=u(p_0)=H_2(Z_\mu(p_0))$. Thus, $H$ is single-valued.
\end{proof}

\begin{theorem}
Suppose $\mu =i\beta\in i\R^\ast$. Let $Z_{i\beta}$ be the first integral of $L$ given by $(3.12)$
with $m=\min \abs{Z_{i\beta}}$ and $M=\max\abs{Z_{i\beta}}$.
Let $\Om$ be an open  subset of $\R^2\baks\{0\}$ containing an annulus
\[
\A (r_1,r_2)=\{ (x,y);\ r_1^2\le x^2+y^2\le r_2^2\}\quad\mathrm{with}\quad
r_2>r_1\exp(2\pi\abs{\beta})\, .
\]
Then, for every $u\in C^0(\Om)$ satisfying $(3.1)$, there exist $H\in C^0(\A (m,M))$,
 $H$ holomorphic in the
interior of $\A(m,M)$ and such that $u=H\circ Z_{i\beta}$.
It follows, in particular, that
$u$ extends as a continuous solution of $(3.1)$ to the punctured plane $\R^2\baks\{0\}$.
\end{theorem}

\begin{proof}
First, we verify that $Z_{i\beta}(\A(r_1,r_2))=\A(m.M)$. Let $\ta_m,\, \ta_M\,\in [0,\ 2\pi]$
be such that $\abs{Z_{i\beta}(r,\ta_m)}=m$ and $\abs{Z_{i\beta}(r,\ta_M)}=M$. Then $Z_{i\beta}$ maps
the segment $r\ei{i\ta_m}$ (respectively, $r\ei{i\ta_M}$) with $r_1\le r\le r_2$ onto the circle
$\abs{z}=m$ (respectively, $\abs{z}=M$). This is because
\[
Z_{i\beta}(r,\ta)=\ei{-\phi_2(\ta)}\exp\left[i\left(\ta+\phi_1(\ta)-\frac{\ln r}{\beta}\right)\right]
\]
and $\ln r_2-\ln r_1 >2\pi\abs{\beta}$. Hence, the image of each segment $r\ei{i\ta_0}$ with
$r_1\le r\le r_2$ is the circle with radius $\ei{-\phi_2(\ta_0)}$. As $\ta_0$ varies from $0$ to
$2\pi$, these radii reach the minimum $m$ and the maximum $M$.

Using an argument similar to that used in the proof of Theorem 3.3, we can show that if
$u\in C^0(\A(r_1,r_2))$ solves $(3.1)$, then $u=H\circ Z_{i\beta}$ with $H$ continuous
in the annulus $\A(m,M)$ and holomorphic in the interior.
\end{proof}

\begin{remark}
{\rm Theorems 3.1, 3.2, and 3.3 imply, in particular, that whenever $u\in C^0(\Om)$, with $0\in\Om$, solves
(3.1), then for every point $p\in\Si\baks\{0\}$,  such that $Z_\mu(p)$ is an interior point of $Z_\mu (\Om)$, then
the solution $u$ is $\cinf$ at p.}
\end{remark}

Now we consider the extension of the solutions of equation (3.1) across the boundary $\pa\Om$
for a class of open sets and for vector fields with $\mu\ge 0$.
A set $\Om\subset\R^2$ is \emph{starlike} with respect to $0\in\Om$ if the segment
$[0,\ p]\subset\Om$ for every $p\in\Om$.

Let $L$ be a homogeneous vector field with number $\mu \ge 0$ and with first integral $Z_\mu$ given
by (2.11) when $\mu>0$ or by (2.7) when $\mu =0$.
Let $\Om$ be open, bounded, starlike with respect to 0, and with boundary  of class $C^1$.
Define an equivalence relation on the boundary $\pa\Om$ by
\[
p\,\sim\, p'\ \Longleftrightarrow\ \arg Z_\mu(p)= \arg Z_\mu(p')\, .
\]
Denote by $\mathrm{cl}(p)$ the equivalence class of $p$. Hence $\mathrm{cl}(p)$ is a closed
subset of $\pa\Om$. Define the function $\rho$ on $\pa\Om$ by
\begeq
\rho\,:\ \pa\Om\, \longrightarrow\, \R^+ ,\quad \rho(p)=\max_{p'\in\mathrm{cl}(p)}
\abs{Z_\mu(p')}\, .
\stopeq
Note that the function $\rho$ is continuous on $\pa\Om$. Let
\begeq
\Om_L=\bigcup_{p\in\pa\Om} [0,\ \ \Lambda (p)\ei{i\arg p})
\stopeq
where
\begeq
\Lambda (p)=\left\{\begar{ll}
\dis \rho(p)^\mu\ei{\mu\phi_2(\arg p)} &\quad \mathrm{if}\ \mu >0\, ;\\
\dis \rho(p)^{1/\sigma}\ei{\phi_2(\arg p)/\sigma} &\quad \mathrm{if}\ \mu =0\, ,
\stopar\right.\stopeq
where $\sigma$ is given in (2.4).
The set $\Om_L$ is starlike with respect to $0$ and $(\Om_L)_L=\Om_L$. We will prove in the next theorem
that $\Om_L$ is the envelope of extendability of $\Om$: every continuous solution of
(3.1) in $\Om$ extends  as a continuous solution to $\Om_L$.

\begin{theorem}
Suppose that $\mu\ge 0$. Let $\Om$ be open, bounded, with $C^1$ boundary, and starlike with
respect to 0, and let $\Om_L$ be given by $(3.7)$. Then
\begin{itemize}
\item[(a)] $\Om_L=Z_\mu^{-1}(Z_\mu(\Om))$
\item[(b)] If $u\in C^0(\Om)$ solves equation $(3.1)$, then there exists
a solution $\widehat{u}\in C^0(\Om_L)$ of $(3.1)$
such that $\widehat{u}=u$ in
$\Om$.
\item[(c)] There exists $v\in C^0(\Om_L)$ solution of $(3.1)$ such that $v$ has no extension
to a larger set.
\item[(d)] If $\mu >0$, then  any solution $\widehat{u}\in C^0(\Om_L)$ of $(3.1)$
belongs to $\cinf(\Om_L\baks\{0\})$.
\item[(e)] If $\mu =0$, then any solution $\widehat{u}\in C^0(\Om_L)$ of $(3.1)$ belongs to
$\cinf(\Om_L\baks Z_0^{-1}(\R))$.
\end{itemize}
\end{theorem}

\begin{proof}
(a) Let $p_0=(r_0,\ta_0)\in\Om_L\baks\{0\}$. The ray through $0$ and $p_0$ intersects $\pa\Om$ at
a point $p'_0=(r'_0,\ta_0)$. Thus, we have
\[
r_0=t\Lambda(p_0')=t\, \left[\rho(p_0')\ei{\phi_2(\ta_0)}\right]^\mu,\quad \mathrm{for\ some}\ 0<t<1\, .
\]
The exponent $\mu$ needs to be replaced by $1/\sigma$ if $\mu=0$.
Let $p_\ast=(r_\ast,\ta_\ast)\in\pa\Om$ be such that $p_\ast\in \mathrm{cl}(p_0')$ and
$\rho(p_0')=\abs{Z_\mu(p_\ast)}$.
Hence,
\[
Z_\mu(p_\ast)=\rho(p_0')\ei{i(\ta_\ast+\phi_1(\ta_\ast))}=\rho(p_0')\ei{i(\ta_0+\phi_1(\ta_0))}\, .
\]
Then,
\[
Z_\mu(p_0)=t^{1/\mu}\, \left[\rho(p'_0)\ei{\phi_2(\ta_0)}\right]
\ei{-\phi_2(\ta_0)}\ei{i(\ta_0+\phi_1(\ta_0))}
=t^{1/\mu}\, Z_\mu(p_\ast)\,.
\]
Since $0<t^{1/\mu}<1$ and $Z_\mu(\Om)$ is starlike with respect to 0, then $Z_\mu(p_0)\in Z_\mu(\Om)$
and $\Om_L\subset Z_\mu^{-1}\left(Z_\mu(\Om)\right)$. A similar argument can be used to show that
 $ Z_\mu^{-1}\left(Z_\mu(\Om)\right)\subset\Om_L$

(b) Suppose that $u\in C^0(\Om)$ satisfies (3.1). Then (Theorem 3.3, or 3.2 when $\mu =0$)
there exists a continuous function $H$ on $Z_\mu(\Om)$
and holomorphic in its interior such that $u=H\circ Z_\mu$. The function
$\widehat{u}=H\circ Z_\mu$, defined on $\Om_L=Z_\mu^{-1}\left(Z_\mu(\Om)\right)$, is the sought
extension of $u$.

(c) Let $H$ be continuous in $Z_\mu(\Om)$, holomorphic in the interior, and such that $H$ has no
holomorphic extension. Then $v=H\circ Z_\mu$ is a solution of (3.1) in $\Om_L$ that can't be
extended.

(d)  If $\mu >0$, then, as in part (a), it can be shown that $\pa Z_\mu(\Om)=Z_\mu(\pa\Om_L)$.
Hence, if $p\in\Om_L\baks\{0\}$, then $Z_\mu(p)$ is an interior point of $Z_\mu(\Om_L)\baks\{0\}$
and since $u=H\circ Z_\mu$,  $H$ holomorphic at $Z_\mu (p)$, $u$ is smooth at $p$.

(e) If $\mu =0$, then $\pa Z_0(\Om)\subset Z_0(\pa\Om_L)\cup (\R+i0)$. As in part (d) if
$p\in \Om_L\baks\{0\}$, $p\not\in Z_0^{-1}(\R)$, then $Z_0(p)$ is an interior point of
$Z_0(\Om)$ and $u$ is $\cinf$ at $p$.
\end{proof}

We will say that the vector field $L$ satisfies the \emph{Liouville property} in $\R^2\baks\{0\}$
if every continuous and bounded solution of $Lu=0$  in $\R^2\baks\{0\}$ is constant.
We have the following characterization.

\begin{theorem}
$L$ satisfies the Liouville property if and only if $\re{\mu}\ne 0$
\end{theorem}

\begin{proof}
Suppose that $\re{\mu}>0$. Then $Z_\mu(R^2)=\C$. It follows that for every $p\in\R^2\baks\{0\}$,
$Z_\mu(p)$ is an interior point of $Z_\mu(\R^2\baks\{0\})$ and consequently, if $u\in C^0(\R^2\baks\{0\})$
solves (3.1), then $u\in\cinf(\R^2\baks\{0\})$ and $u=H\circ Z_\mu$ with $H$ an entire function.
If, in addition, $u$ is bounded, then $H$ is constant.

If $\mu =0$, then $Z_0(\R^2)$ is the upper half plane in $\C$. If $H$ is any nonconstant, bounded,
holomorphic function defined in the upper half plane, then $u=H\circ Z_0$ is nonconstant, bounded
and smooth solution of (3.1) in $\R^2$. If $\mu =i\beta$ with $\beta\ne 0$, the same argument
applies by replacing the half plane by the annulus $\A(m,M)$.
\end{proof}

\section{Examples}

\noindent{\bf 1.} Let
\[
L=r^{\lam-1}L_0=r^{\lam-1}\left[\pt -\left(\cos\ta+i\frac{\pi}{2}\sin\ta\right)r\pr\right]\, .
\]
We have $\mu =0$ and the first integral is $\dis Z_0=r\ei{\sin\ta}
\exp{\left[i\frac{\pi-\pi\cos\ta}{2}\right]}$.
The characteristic set $\Si$ consists of the rays $\ta =0$ and $\ta =\pi$. The image of the
unit disc $D$  is the region in the upper half plane $\mathrm{Im}z \ge 0$ whose
boundary consists of the curve $\dis z=\ei{\sin\ta}
\exp{\left[i\frac{\pi-\pi\cos\ta}{2}\right]}$ for $0\le \ta\le \pi$ and the segment $[-1,\ 1]$.
\begin{figure}[ht]
 \centering
\scalebox{0.27} {\includegraphics{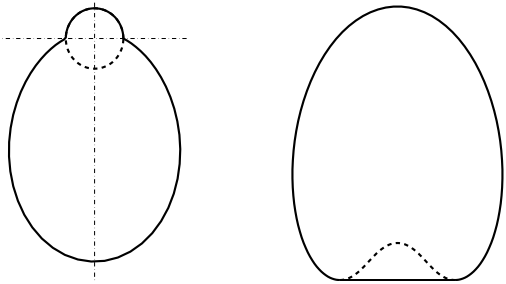}}
\caption{\textit{The unit disc, its envelope, and their image}}
\end{figure}
Note the lower semicircle $x^2+y^2=1$, $y < 0$, is mapped via $Z_0$ into the interior of
$Z_0(D)$. Hence, any continuous solution in $D$ extends through the lower semicircle.
In fact,  the function $\rho$ (defined in (3.6)) in this case is given by
$\rho(\ta)=\ei{\abs{\sin\ta}}$. The envelope $D_L$ of $D$ is the interior of the curve
given by $\ei{\left(\abs{\sin\ta}-\sin\ta\right)}\,\ei{i\ta}$ (see Figure 1.)

\vspace{.2cm}

\noindent{\bf 2.} Let $k\in\Z^+$,  $\mu\in \R^++i\R$, and
\[
L=r^{\lam-1}\left[ \pt-i\mu (1+k\cos k\ta)r\pr\right]\, .
\]
The first integral is $Z_\mu =r^{1/\mu}\ei{i(\ta+\sin(k\ta))}$ and the characteristic set
$\Si$ consists of the $2k$ rays given by the roots of the equation $1+k\cos k\ta =0$.
Note that $L$ does not satisfy $\condp$ at any point of $\Si$. The image, via $Z_\mu$, of the unit disc
$D$ is the unit disc. If follows that if $u\in C^0(D)$ solves $Lu=0$, then $u\in \cinf(D\baks\{0\})$.

\vspace{.2cm}

\noindent{\bf 3.} Let
\[
L=r^{\lam-1}\left[\pt-i\left(2\cos 2\ta-2\sin 4\ta+i\right)r\pr\right]\, .
\]
We have $\mu =i$ and the first integral is
\[
Z_i =r^{1/i}\exp\left[ \sin 2\ta +\frac{1}{2}\cos 4\ta\right]\, \ei{i\ta}\, .
\]
\begin{figure}[ht]
 \centering
\scalebox{0.27} {\includegraphics{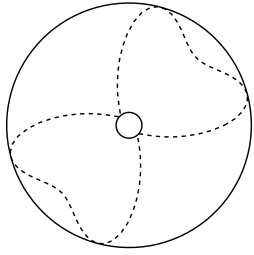}}
\caption{\textit{The annulus $\A(\ei{-3/2},\ei{3/4})$ and the image of the unit circle
(dotted curve)}}
\end{figure}

The characteristic set $\Si$ consists of eight rays:
\[
\ta =\frac{\pi}{4};\ \frac{3\pi}{4};\ \frac{5\pi}{4};\ \frac{7\pi}{4};\
\frac{\pi}{12};\ \frac{5\pi}{12};\ \frac{13\pi}{12};\ \frac{17\pi}{12}\, .
\]
The vector field $L$ does not satisfy $\condp$ on any point of $\Si$. The function $\abs{Z_i}$ reaches
a maximum $\ei{3/4}$ on the four rays
$\dis\ta=\frac{\pi}{12};\ \frac{5\pi}{12};\ \frac{13\pi}{12};\ \frac{17\pi}{12}$;
and it reaches a minimum $\ei{-3/2}$ on the two rays $\dis \ta =\frac{3\pi}{4};\
\frac{7\pi}{4}$. We have then $Z_i(\R^2\baks\{0\})=\A(\ei{-3/2},\ei{3/4})$.
Hence, any $C^0$ solution of $Lu=0$ in a region containing an annulus $\A (r_1,r_2)$
with $r_2>r_1\ei{2\pi}$ extends to a continuous solution $\widehat{u}$ in $\R^2\baks\{0\}$.
Furthermore, $\widehat{u}$ is $\cinf$ everywhere except possibly on the rays where $\abs{Z_i}$
reaches its extreme values. In particular, $\widehat{u}$ is $\cinf$ on the rays
$\dis\ta=\frac{\pi}{4};\ \frac{5\pi}{4}$ even though $L$ does not satisfy $\condp$ along
these rays.

\section{The equation $Lu=f$ with $f$ homogeneous}
In this section, we consider the equation
\begeq
Lu=f
\stopeq
with $f$ a homogeneous function of degree $\sigma \in \C$. We have following
theorem

\begin{theorem} Let $l\ge0$
and $f\in C^l(\R^2\baks\{0\})$ be $\sigma$-homogeneous with
$\re{\sigma}>0$. Let $\mu =\mu(L)$ and assume that
\begeq
\mu (\lam-\sigma -1)\notin\Z \, .
\stopeq
Then equation $(5.1)$ has a distribution solution $u\in\distr(\R^2) $, with
$u$ homogeneous with degree $\sigma +1-\lam$ and such that
$u\in C^{l+1}(\R^2\baks\{0\})$. In particular, if $\re{\sigma}>\re{\lam}-1$,
then $u$ is H\"{o}lder continuous at 0.
\end{theorem}

\begin{proof}
Since $f$ is homogeneous, then $f(r,\ta)=r^\sigma f_0(\ta)$ with $f_0\in C^l(\cir)$.
Let
\begeq
\psi(\ta)=\int_0^\ta\frac{q(s)}{p(s)}\, ds
\stopeq
so that $\psi(2\pi)=2\pi\mu$. The function
\begeq
v(\ta)=\left[K+\int_0^\ta\frac{f_0(s)}{p(s)}\exp\left(-i(\sigma+1-\lam)\psi(s)\right)
\, ds\right]\exp\left(i(\sigma+1-\lam)\psi(\ta)\right)
\stopeq
with $K\in\C$ is the general solution of the following ODE
\begeq
p(\ta)v'(\ta)-i(\sigma+1-\lam)q(\ta)v(\ta)=f_0(\ta)\, .
\stopeq
Since, $\ei{2\pi i\mu(\sigma+1-\lam)}\ne 1$  (by (5.2) and $\psi(2\pi)=2\pi \mu$),
then the only periodic solution of (5.5) is obtained when
\begeq
K=\frac{1}{1-\ei{2\pi i\mu(\sigma+1-\lam)}}
\int_0^{2\pi}\frac{f_0(s)}{p(s)}\exp\left(-i(\sigma+1-\lam)\psi(s)\right)\, ds\, .
\stopeq

With this value of $K$, let
\begeq
u(r,\ta)=r^{\sigma+1-\lam}v(\ta)\, .
\stopeq
Then, $u\in C^{l+1}(\R^2\baks\{0\})$ and it is H\"{o}lder continuous at 0 if $\re{\sigma}>\re{\lam}-1$.
A simple calculation shows that $u$ solves (5.1) in $\R^2\baks\{0\}$. Now we show that $u$ solves
(5.1) in $\R^2$ in the sense of distributions. Observe that when $L$ is given by (1.2), then its transpose is
\begeq\begar{ll}
^tL & =-(L+\mathrm{div}L)\\ \\
& =-r^{\lam-1}\left[p(\ta)\pt -iq(\ta)r\pr -i(\lam+1)q(\ta)+p'(\ta)\right]\, .
\stopar\stopeq
Consider the case $\re{\sigma}>\re{\lam}-2$, so that $u\in L^1(\R^2)$. Then, for $\Phi\in C_0^\infty (\R^2)$
a test function, we have
\begeq\begar{ll}
\< Lu,\Phi\> & = \<u,\, ^tL\Phi\>\\
&=\dis -\int_0^\infty\!\!\!\int_0^{2\pi}\!\!\!r^{\sigma+1}v
\left(p\Phi_\ta-irq\Phi_r-i(\lam+1)q\Phi +p'\Phi\right)\, d\ta dr\, .
\stopar\stopeq
We use integration by parts together with (5.5) to get
\begeq\begar{ll}
\dis\int_0^{2\pi}\!\!\! p(\ta)v(\ta)\Phi_\ta d\ta & =\dis
-\int_0^{2\pi}\!\!\!
\left[ p'v+i(\sigma+1-\lam)qv+f_0\right]\Phi\, d\ta\\ \\
\dis\int_0^\infty\!\!\! r^{\sigma+2}\Phi_rdr&=\dis-\int_0^\infty\!\!\!
 (\sigma+2)r^{\sigma+1}\Phi dr\, .
\stopar\stopeq
After using (5.10) in (5.9), we obtain
\begeq
\<Lu,\Phi\> =\int_0^\infty\!\!\!\int_0^{2\pi}\!\!\!
r^{\sigma+1}f_0(\ta)\Phi(r,\ta)d\ta dr =\<f,\Phi\>
\stopeq
The theorem is proved in this case.

Next, suppose that $0<\re{\sigma}\le\re{\lam}-2$. Let $k=[\re{\lam-\sigma}]$,
where $[x]$ denotes the greatest integer $\le x$. Define $u\in\distr(\R^2)$ by
\begeq
\<u,\Phi\>=C_0\int_0^\infty\!\!\!\int_0^{2\pi}\!\!\!
r^{k-(\lam-\sigma)}v(\ta)\dd{^{\, k-2}\Phi}{r^{k-2}}\, d\ta dr\, ,
\quad \Phi\in\ccinf(\R^2)
\stopeq
where
\[
C_0=\frac{\Gamma(\lam-\sigma-k)}{\Gamma(\lam-\sigma-2)}
\]
and $\Gamma$ is the gamma function. Since we are using  derivatives with the polar coordinate,
we need  to verify that (5.12) indeed
defines  a distribution.
We have
\[
\dd{^{\, n}}{r^n}=\sum_{j=0}^n\Big(\,^n_j\Big)\cos^{n-j}\ta\sin^j\ta\dd{^{\,n}}{x^{n-j}\pa y^j}
\]
(this formula can be established by induction), and for a compact set $K\subset\R^2$ with
$\supp\Phi \subset K$, we have
\[
\abs{\<u,\Phi\>} \,\le\, C(K)\, \sum_{j=0}^{k-2}\sup_K\left|\dd{^{k-2}\Phi}{x^{k-2-j}\pa y^j}\right|
\]
where $C(K)$ is a constant depending only on $K$, $v$, $\sigma$, and $\lam$. To verify the homogeneity
of $u$, define the distribution $u_s\in\distr(\R^2)$ (with $s>0$) by
$\<u_s,\Phi\>=s^{-2}\<u,\Phi (x/s,y/s)\>$. We have
\[
\dd{^{\,k-2}}{r^{k-2}}\left(\Phi(\frac{r}{s},\ta)\right)
=\frac{1}{s^{k-2}}\dd{^{k-2}\Phi}{r^{k-2}}\left(\frac{r}{s},\ta\right)
\]
and the change of variable $\rho =r/s$ in the integral defining $u_s$ gives
\[\begar{ll}
\<u_s,\Phi\> & =\dis s^{-2}C_0\int_0^\infty\!\!\!\int_0^{2\pi}\!\!\!
s^{\sigma-\lam+3}\rho^{k-(\lam-\sigma)}v(\ta)
\dd{^{\, k-2}\Phi}{\rho^{k-2}}(\rho,\ta)\, d\ta d\rho\\
& =s^{\sigma-\lam+1}\<u,\Phi\>\, .
\stopar\]
This shows the homogeneity of $u$. Now we verify that $u$ solves (5.1).
The function $^tL\Phi$ has compact support and vanishes to order $\re{\lam}-1$ at the origin
($r=0$). We can use integration by parts to obtain
\[\begar{ll}
\<Lu,\Phi\> & = \<u,\, ^tL\Phi\> =\dis
C_0\int_0^\infty\!\!\!\int_0^{2\pi}\!\!\!
r^{k-(\lam-\sigma)}v\dd{^{\, k-2}{\, ^tL\Phi}}{r^{k-2}}\, d\ta dr\\
& =\dis C_1\int_0^\infty\!\!\!\int_0^{2\pi}\!\!\!
r^{k-(\lam-\sigma)-1}v\dd{^{\, k-3}{\, ^tL\Phi}}{r^{k-3}}\, d\ta dr
\stopar\]
with
\[
C_1=(\lam-\sigma-k)C_0=\frac{\Gamma(\lam-\sigma-k+1)}{\Gamma(\lam-\sigma-2)}\, .
\]
In general after $j$ integrations by parts (with $j\le k-2$) we get
\[
\<Lu,\Phi\>=\frac{\Gamma(\lam-\sigma-k+j)}{\Gamma(\lam-\sigma-2)}\,
\int_0^\infty\!\!\!\int_0^{2\pi}\!\!\!
r^{k-(\lam-\sigma)-j}v\dd{^{\, k-2-j}{\ ^tL\Phi}}{r^{k-2-j}}\, d\ta dr\, .
\]
In particular, for $j=k-2$, we obtain
\[
\<Lu,\Phi\> =\int_0^\infty\!\!\!\int_0^{2\pi}\!\!\!
r^{\sigma-\lam+2}v(\ta)\, ^tL\Phi \, d\ta dr\, .
\]
Finally, a calculation similar to that carried out in the case $\re{\sigma}>\re{\lam}-2$, gives
\[
\<Lu,\Phi\> =\int\!\!\int_{\R^2}f(x,y)\Phi(x,y)dxdy\, .
\]
\end{proof}

When (5.2) does not hold, then equation (5.1) can still be solved, provided that $f$
satisfies a compatibility condition.

\begin{theorem}
Suppose that $\mu(\sigma+1-\lam)\in\Z$. If $f(r,\ta)=r^\sigma f_0(\ta)$ satisfies
\begeq
\int_0^{2\pi}\!\!\!\frac{f_0(s)}{p(s)}\exp\left[-i\mu(\sigma+1-\lam)\psi(s)\right]\, ds
=0\, ,
\stopeq
where $\psi$ is given by $(5.3)$,
then equation $(5.1)$ has a homogeneous distribution solution as that given in Theorem $5.1$.
\end{theorem}

The proof is identical to that of Theorem 5.1. This time we define $v$ by (5.4) with constant
$K=0$.

Now, for a class of vector fields,
 we consider equation (5.1) when $f=\delta$,  the Dirac distribution.
 More precisely, we have the following result.

\begin{theorem}
Suppose that $\mathrm{div}L=0$; the homogeneity degree of $L$ is $\lam=N\in \Z^+$; and
that $\mu\ne 0$. Then, the distribution $u\in\distr(\R^2)$ given by
\begeq
\<u,\Phi\> =\frac{1}{2\pi i\mu(N-1)!}
\int_0^\infty\!\!\!\int_0^{2\pi}\!\!\!\frac{\ln r}{p(\ta)}\dd{\,^N\Phi}{r^N}\, d\ta dr\,,\quad
\Phi\in\ccinf(\R^2)
\stopeq
satisfies the equation $Lu=\delta$.
\end{theorem}

\begin{proof}
For a test function $\Phi$, we have $r^N\Phi =O(r^N)$ as $r\,\longrightarrow\, 0$,
and integration by parts gives
\[
\int_0^\infty\!\!\!
\ln r\,\dd{\,^N(r^N\Phi)}{r^N}\, dr =-
\int_0^\infty\!\!\! r^{-1}\,\dd{\,^{N-1}(r^N\Phi)}{r^{N-1}}\, dr\, .
\]
By repeated integration by parts, $N$ times, we get
\begeq
\int_0^\infty\!\!\!
\ln r\,\dd{\,^N(r^N\Phi)}{r^N}\, dr =-(N-1)!\int_0^\infty\!\!\!
 \Phi(r,\ta)\, dr\, .
\stopeq
Since $\mathrm{div}L=0$, then  $\<Lu,\Phi\>=-\<u,L\Phi\> $ and
\begeq
\<Lu,\Phi\> =\dis\frac{-1}{2\pi i\mu(N-1)!}\!\int_0^\infty\!\!\!\!\int_0^{2\pi}
\!\!\!\ln r\dd{\,^N(r^{N-1}\Phi_\ta)}{r^N}-i\frac{q(\ta)}{p(\ta)}\ln r
\dd{\,^N(r^{N}\Phi_r)}{r^N}\, d\ta dr
\stopeq
We have
\begeq
\int_0^{2\pi}\!\!\dd{\,^N(r^{N-1}\Phi_\ta)}{r^N}\, d\ta=
\int_0^{2\pi}\!\!\left(\dd{\,^N(r^{N-1}\Phi)}{r^N}\right)_\ta\, d\ta \, =0\, .
\stopeq
Expressions  (5.15) and (5.17) give
\begeq
\int_0^\infty\!\!\!\ln r\dd{\,^N(r^N\Phi_r)}{r^N}\, dr =-(N-1)!\int_0^\infty
\!\!\!\Phi_r (r,\ta)dr =(N-1)!\Phi (0)\, .
\stopeq
After using (5.17) and (5.18), (5.16) becomes
\[
\<Lu,\Phi\>=\frac{1}{2\pi\mu}\int_0^{2\pi}\frac{q(\ta)}{p(\ta)}\Phi(0)d\ta =\Phi(0)\, .
\]
\end{proof}

\section{Equation $Lu=f$ with $f$ real analytic}
In this section, we consider the equation
\begeq
Lu=f
\stopeq
with $f$ real analytic in a neighborhood of $0\in\R^2$. We show that this equation
has a distribution solution which is smooth outside  0 when the numbers
$\mu$ and $\lam$ attached to $L$ satisfy a certain diophantine condition.

Let $\mu,\,\lam\, \in\C$ be the numbers associated with $L$, with $\re{\mu}\ge 0$ and $\re{\lam}>1$.
We say that the pair
$(\mu,\lam)$ is \emph{resonant} if $\mu\lam\in\mu\Z^++\Z$. That is, if there exist
$l,k\in \Z$, with $l\ge 1$, such that $\mu\lam =\mu l+k$. We call such a number
 $l$ a \emph{resonant integer}
  for $(\mu,\lam)$. Denote by $\mathbb{J}(\mu,\lam)$ the set of resonant integers
for $(\mu,\lam)$.
Note that when $f$ is a homogeneous function of degree $(j-1)$ and $j$ is a resonant integer for
$(\mu,\lam)$, then the equation $Lu=f$ does not have a homogeneous solution unless $f$
satisfies a compatibility condition.

\begin{lemma}
Suppose that $(\mu,\lam)$ is resonant. Then
\begin{itemize}
\item[1.] If $(\mu,\lam)\not\in \Q^+\times\Q^+$, then
$(\mu,\lam)$ has a unique resonant integer.
\item[2.] If $\mu\in\Q^+$ or if $\lam\in\Q^+$, then
$(\mu,\lam)$ has infinitely many resonant integers.
\end{itemize}
\end{lemma}

\begin{proof}
First note that since $(\mu,\lam)$ is resonant, then if one of the numbers $\mu$ or $\lam$
is in $\Q^+$, then so is the other. Let $(j_0,k_0)\in\Z^+\times\Z$ be such that
$\mu\lam=\mu j_0+k_0$. If $(j_1,k_1)\in\Z^+\times\Z$ is any other pair satisfying
$\mu\lam=\mu j_1+k_1$, then $\mu (j_1-j_0)=k_0-k_1$ and $\mu\in\Q$.
Conversely, if $\mu =m/n$ with $m,n\in\Z^+$ and $\gcd (m,n)=1$. $j=j_0+ln$,
$k=k_0-lm$ satisfies $\mu\lam=\mu j+k$ for any $l\in\Z$.
\end{proof}

\begin{lemma}
Suppose that $\mathrm{div}L=0$. The associated pair  $(\mu,\lam)$ is
resonant  if and only if either $\mu=0$ or if
$\lam\in\Q^+$. In this case the pair has infinitely many resonant integers.
In particular, if in addition $L$ is real analytic at 0 with homogeneous degree
$\lam=N\in\Z^+$, then $(\mu,N)$ is resonant with infinitely many resonant integers.
\end{lemma}

\begin{proof}
Since $\mathrm{div}L=0$, then $\mu =p/(\lam+1)$ for some $p\in\Z$, $p\ge 0$ (Remark 2.1).
If $p=0$, the pair $(0,\lam)$ is trivially resonant. Suppose that $p\ge 1$. If $(\mu,\lam)$ is
resonant, then there exists $(l_0,k_0)\in \Z^+\times\Z$ such that
$p\lam/(\lam+1)=pl_0/(\lam+1)\, +k_0$. Therefore, $\lam (p-k_0)=pl_0+k_0$ and $\lam\in\Q^+$.
Conversely, if $\lam=m/n$ with $m,n\, \in\Z^+$ and $\gcd (m.n)=1$, then
$\mu =p/(\lam+1)=pn/(m+n)$. It can be verified that for every $s\in \Z^+$, $j=s(m+n)-1$
is $(\mu,\lam)$-resonant with a corresponding $k=p(1-sn)$.
\end{proof}

Now we introduce a  condition $\mcaldc$ for nonresonant pairs $(\mu,\lam)$.
\begin{itemize}
\item[$\mcaldc$]: There exists $C>0$ such that for every $j\in\Z^+$
\[
\left|1-\ei{2\pi i\mu(j-\lam)}\right| \ge C^j
\]
\end{itemize}

\begin{proposition}
Suppose that $\mu\in\C$, $\mu\ne 0$, $\re{\mu}\ge 0$ and $\lam\in \C$, are such that
$(\mu,\lam)$ is nonresonant. Then,
\begin{itemize}
\item[(a)] If $\mu\notin\R$, then $(\mu,\lam)$ satisfies condition $\mcaldc$.
\item[(b)] If $\mu\in\R^+$ and $\lam\notin\R$, then
$(\mu,\lam)$ satisfies condition $\mcaldc$.
\item[(c)] If $\mu,\, \lam\, \in \R^+$, then condition$\mcaldc$
holds if and only if the following condition holds
\begin{itemize}
\item[$\mcaldcc$]: There exists $C>0$ such that for every $j\in\Z^+$ and $k\in\Z$
\[
\left|\mu(j-\lam)-k\right| \ge C^j
\]
\end{itemize}
\end{itemize}
\end{proposition}

\begin{proof}
Let $\mu=\mu_1+i\mu_2$, $\lam=\lam_1+i\lam_2$ with
$\mu_1,\mu_2,\lam_1,\lam_2\in\R$ and $\mu_1,\lam_1\ge 0$. Then for
$j\in\Z^+$ we have
\begeq
\ei{2\pi i\mu(j-\lam)}=\ei{2\pi (\mu_1\lam_2+\mu_2\lam_1-\mu_2j)}\,
\ei{2\pi i(\mu_1j-\mu_1\lam_1+\mu_2\lam_2)}
\stopeq
and since $(\mu,\lam)$ is nonresonant, then for every $j\in\Z^+$, $k\in\Z$,
\begeq
(\mu_1\lam_1-\mu_2\lam_2-j\mu_1-k)+i(\mu_1\lam_2+\mu_2\lam_1-j\mu_2)\,\ne 0
\stopeq

Suppose that $\mu_2\ne 0$. If for every $j\in\Z^+$
$\mu_1\lam_2+\mu_2\lam_1-j\mu_2 \ne 0$, then
\begeq
\min_{j\in\Z^+}\left|1-\ei{2\pi(\mu_1\lam_2+\mu_2\lam_1-j\mu_2)}\right|\, >0\,.
\stopeq
Consequently the pair $(\mu,\lam)$ satisfies $\mcaldc$.
If there exists $j_0\in\Z^+$ such that
$\mu_1\lam_2+\mu_2\lam_1-j_0\mu_2 =0$, then  $\left|\ei{2\pi i\mu(j_0-\lam)}\right|=1$.
It follows from (6.3) that, for every $k\in\Z$, we necessarily have
$\mu_1\lam_1-\mu_2\lam_2-j_0\mu_1-k\ne 0$. Hence,
\begeq
\left|\ei{2\pi i\mu(j_0-\lam)}-1\right| >0
\stopeq
and we also have
\begeq
\min_{j\in\Z^+,\, j\ne j_0}\left|1-\ei{2\pi(\mu_1\lam_2+\mu_2\lam_1-j\mu_2)}\right|\, >0\,.
\stopeq
Inequalities (6.5), (6.6) imply that the pair $(\mu,\lam)$ satisfies $\mcaldc$,
which proves part (a).

Next, suppose that $\mu=\mu_1\in\R^+$ and $\lam_2\ne 0$. Then for every $j\in\Z$,
$\left|\ei{2\pi i\mu_1(j-\lam)}\right|=\ei{2\pi\mu_1\lam_2} \ne 1$ and part (b) follows.

To prove part (c), suppose that $\mu=\mu_1>0$ and $\lam=\lam_1>0$. We need to show
that conditions $\mcaldc$ and $\mcaldcc$ are equivalent.
If $\mcaldcc$ does not hold, then for every $l\in\Z^+$, there exists $j_l\in\Z^+$
and $k_l\in\Z$ such that
\begeq
\left|\mu(j_l-\lam)-k_l\right|\, <\, l^{-j_l}\, .
\stopeq
We have then
\begeq\begar{ll}
\dis\left|\ei{2\pi i\mu(j_l-\lam)}-1\right|^2 & =\dis\left|\ei{2\pi i\left[\mu(j_l-\lam)-k_l\right]}-1\right|^2\\
\\
& =\dis 2\left(1-\cos\left[2\pi(\mu(j_l-\lam)-k_l)\right]\right)\\ \\
& =\dis 4\pi\left[\mu(j_l-\lam)-k_l\right]\,\sin\ta_l
\stopar\stopeq
for some $\ta_l$. This, together with (6.7), give
\begeq
\left|\ei{2\pi i\mu(j_l-\lam)}-1\right|^2\, <\, 4\pi l^{-j_l}
\stopeq
which means that $\mcaldc$ does not hold.

Conversely, suppose that $\mcaldc$ does not hold. Then for every $l\in \Z^+$, there exists
$j_l\in\Z^+$ such that
\begeq
\left|\ei{2\pi i\mu(j_l-\lam)}-1\right|\, <\, l^{-j_l}\, .
\stopeq
Let $k_l=[\mu(j_l-\lam)]$.
Since $0\le \mu(j_l-\lam)-k_l<1$, it follows from (6.10) that
$\dis\lim_{l\to\infty}(\mu(j_l-\lam)-k_l)=0$. Consequently, after using
$\abs{\ei{i\alpha}-1}^2=2(1-\cos\alpha)>\alpha^2$  (for $\alpha$ small), we obtain
\[
\dis l^{-2j_l} \, >\, \dis\left|\ei{2\pi i(\mu(j_l-\lam)-k_l)}-1\right|^2\, > \, \dis 4\pi \left|\mu(j_l-\lam)-k_l\right|^2
\]
and therefore $\mcaldcc$ does not hold.
\end{proof}

\begin{theorem}
Let $L$ be a $\lam$-homogeneous vector field with $\mu\ne 0$ and such that the pair $(\mu,\lam)$
is nonresonant and satisfies condition $\mcaldc$. Then for every function $f(x,y)$ that is
real analytic at
$0\in\R^2$, there exist $\ep>0$ and a bounded function $w\in \cinf(D(0,\ep)\baks\{0\})$ such that
$u=w/r^{\lam-1}$ is a distribution solution of $(6.1)$ in the disc $D(0,\ep)$.
\end{theorem}

\begin{proof}
We expand the real analytic function $f$ as
\begeq
f(x,y)=\sum_{j=0}^\infty P_j(x,y)\,\ \ \mathrm{with}\ \
P_j(x,y)=\sum_{k+l=j}\frac{1}{k!l!}\dd{\,^jf}{x^k\pa y^l}(0)\, x^ky^l\, .
\stopeq
Let
\begeq
P_j(x,y)=r^jf_j(\ta)\ \ \mathrm{with}\ \
f_j(\ta)=\sum_{k+l=j}\frac{1}{k!l!}\dd{\,^jf}{x^k\pa y^l}(0)\, \cos^k\ta\sin^l\ta
\stopeq
Let $R>0$ be such that the function $f(x,y)$ has a holomorphic extension $\widehat{f}(\widehat{x},\widehat{y})$
in an open neighborhood of the bidisc $D(0,R)^2\subset\C^2$. Let $\dis M_0=\max_{\abs{\widehat{x}},\,\abs{\widehat{y}}\,\le R}
\left|\widehat{f}(\widehat{x},\widehat{y})\right|$. It follows from the Cauchy integral formula that
\[
\left|\dd{\, ^jf}{x^k\pa y^l}(0)\right|\,\le\, k!l!\, \frac{M_0}{R^j}\, .
\]
Hence, we get the estimate
\begeq
\left|f_j(\ta)\right|\,\le\, (j+1)\frac{M_0}{R^j}\qquad\forall\ta\, .
\stopeq
Let
\begeq
\psi(\ta)=\int_0^\ta\frac{q(s)}{p(s)}\, ds=\psi_1(\ta)+i\psi_2(\ta)
\stopeq
so that $\psi(2\pi)=2\pi\mu$. Define $v_j(\ta)\in C^\infty(\R^2\baks\{0\})$ by
\begeq
v_j(\ta)=\left[K_j+\int_0^\ta\!\!\frac{f_j(s)}{p(s)}\,\ei{-i(j+1-\lam)\psi(s)}ds\right]
\ei{i(j+1-\lam)\psi(\ta)}
\stopeq
with
\begeq
K_j=\frac{1}{1-\ei{2\pi i\mu(j+1-\lam)}}\int_0^{2\pi}\!\!
\frac{f_j(s)}{p(s)}\,\ei{-i(j+1-\lam)\psi(s)}ds\, .
\stopeq
Note that the denominator in $K_j$ is not zero since the pair $(\mu,\lam)$ is nonresonant.
The function $v_j$ is a periodic solution of the differential equation
\begeq
p(\ta)v'(\ta)-i(j+1-\lam)q(\ta)v(\ta)=f_j(\ta)\, .
\stopeq
Therefore $r^{j+1-\lam}v_j(\ta)$ is the homogeneous distribution solution of $Lu=P_j$
given in Theorem 5.1. Consider the series
\begeq
w(r,\ta)=\sum_{j=0}^\infty v_j(\ta)r^j\,.
\stopeq
We prove that $w$ is a $\cinf$ function of $(r,\ta)$ for small $r$. Let $\lam_1$ and $\lam_2$ be
the real and imaginary parts of $\lam$. Then
\[
\left| \ei{-i(j+1-\lam)\psi(s)}\right|=\ei{-\lam_1\psi_2(s)-\lam_2\psi_1(s)}
\left(\ei{\psi_2(s)}\right)^{j+1}\, .
\]
Set
\[
\dis M_1=\max_{0\le s\le 2\pi}\ei{-\lam_1\psi_2(s)-\lam_2\psi_1(s)}\,;\ \
 \dis M_2=\max_{0\le s\le 2\pi}\ei{\psi_2(s)}\, ;\ \ \mathrm{ and}\ \
\dis p_0=\min_{0\le s\le 2\pi}\abs{p(s)}\, .
\]
 Then,
it follows from (6.13), that for $0\le\ta\le 2\pi$, we have
\begeq
\left|\int_0^\ta\!\!\frac{f_j(s)}{p(s)}\ei{-i(j+1-\lam)\psi(s)}\, ds\right|\,\le\,
\frac{2\pi M_1M_0(j+1)M_2^{j+1}}{p_0R^j}\,\le\, C_0^j
\stopeq
for some constant $C_0$ depending only on $f$, $\lam$, $p$ and $q$.
Since $(\mu,\lam)$ satisfies the diophantine condition $\mcaldc$,
it follows from (6.16) and (6.19) that
\begeq
\abs{K_j}\,\le \, C_1^j\,\ \ \mathrm{with}\ \ C_1=\frac{C_0}{C}\, .
\stopeq
This inequality, together with (6.15) and (6.19), imply that
\begeq
\left|v_j(\ta)\right|\,\le\, (C_1^j+C_0^j)M_1M_2^{j+1}\,\le\, C_2^j
\stopeq
with a constant $C_2>0$ depending only on $f$, $\lam$, $p$ and $q$. This proves at once
that the series defining $w$ given by (6.18) converges uniformly for $\ta\in [0,\ 2\pi]$ and
$r\in [0,\ R_0]$ for any positive number $R_0<1/C_2$. Consequently, $w(r,\ta)$ is continuous,
$2\pi$-periodic in $\ta$ and real analytic in $r$. The uniform convergence of the series of
derivatives can be proved by a similar argument. For instance, since
\[
v'_j(\ta)=i(j+1-\lam)\frac{q(\ta)}{p(\ta)}v_j(\ta)+\frac{f_j(\ta)}{p(\ta)}
\]
then similar estimates show that $\sum v'_j(\ta)r^j$ is again uniformly convergent
and so $w\in C^1$ in the $(r,\ta)$ variables. This argument can be repeated for the successive
derivatives leading to $w\in C^\infty ([0,\ R_0]\times\cir)$.

Let $u=w/r^{\lam-1}$. We can write
\begeq
u(r,\ta)=\sum_{j=0}^{[\re{\lam}]}\!v_j(\ta)r^{j+1-\lam}+
\sum_{j=[\re{\lam}]+1}^\infty\!\! v_j(\ta)r^{j+1-\lam}\, .
\stopeq
The finite sum in (6.22) is a distribution solution of the equation\\
$ Lg=\sum_{j=0}^{[\re{\lam}]}P_j\,$ and the infinite sum, which is a function of at least
class $C^1$  at $0\in\R^2$, solves the equation
 $Lg=f-\sum_{j=0}^{[\re{\lam}]}P_j$.
 This completes the proof of the theorem.
\end{proof}

\begin{remark}
{\rm Diophantine conditions such as the one used here appear in connection with the hypoellipticity
and solvability. For instance in {\cite{Ber-Mez1}}, {\cite{Ber-Mez2}}, the analytic solvability
in a neighborhood of a degeneracy curve of a vector field is controlled by a Diophantine condition. }
\end{remark}

\begin{theorem}
Suppose that $(\mu,\lam)$ is nonresonant, satisfies $\mcaldc$, and that $\re{\mu}>0$.
Let $f$ be a real analytic function in a neighborhood of $0\in\R^2$. If $u$ is a distribution
solution of $(6.1)$ such that $u$ is continuous outside 0, then there exists $\ep>0$ such that
$u\in\cinf(D(0,\ep)\baks\{0\})$.
\end{theorem}

\begin{proof}
By Theorem 6.1, there exists a distribution solution $u_0$ of (6.1) with $u_0\in \cinf(D(0,R)\baks\{0\})$.
If $u_1$ is any other solution, continuous outside 0, then $u_0-u_1$ satisfies $L(u_0-u_1)=0$ and
 is continuous outside 0. Since, $\re{\mu}>0$, then it follows from Theorem 3.3, that there exists
$\ep>0$ such that $u_0-u_1\in \cinf(D(0,\ep)\baks\{0\})$.
\end{proof}

\begin{remark}
{\rm If $L$ is real analytic in $\R^2\baks\{0\}$, then the distribution solution
constructed in Theorem 6.1 is real analytic in a punctured neighborhood of 0 and so is
every solution if $\re{\mu}>0$.}
\end{remark}

Now we consider the case when $(\mu,\lam)$ is resonant. As was observed in section 5,
the equation (6.1) with a homogeneous right hand side does not always have a homogeneous
solution unless a compatibility condition is satisfied. A real analytic function $f(x,y)$
is $(\mu,\lam)$-compatible if, its  series expansion (6.11) satisfies
\begeq
\int_0^{2\pi}\!\!\frac{f_j(\ta)}{p(\ta)}\ei{-i(j+1-\lam)\psi(\ta)}\, d\ta  =0\, ,\quad
\forall j\in\mathbb{J}(\mu,\lam)
\stopeq
where $f_j$ is given by (6.12).
Note that when $\mu\notin\Q$, then there is only one resonant integer and (6.23) reduces
to a single condition, while if $\mu\in\Q$, then there are infinitely many conditions
(see Lemma 6.1). In particular, if, in addition, $L$ is real analytic at $0\in\R^2$ and
$\mathrm{div}L=0$, then
there are infinitely many compatibility conditions (see Lemma 6.2).

\begin{theorem}
Suppose that $(\mu,\lam)$ is resonant and satisfies $\mcaldc$.
Then for every $(\mu,\lam)$-compatible, real analytic function $f$, equation $(6.1)$
has a distribution solution as in Theorem $6.1$
\end{theorem}

\begin{remark}
{\rm Among the results contained in  the recent papers {\cite{Tre2}}, {\cite{Tre3}}
are the solvability and hypoellipticity of vector fields
\[
\left(a_{11}x+a_{12}y\right)\px+\left(a_{21}x+a_{22}y\right)\py
\]
with $a_{ij}\in\C$. In particular, such vector fields are not hypoelliptic at 0. This
lack of hypoellipticity is generalized in {\cite{Mez-Sin}} in the real analytic category to
vector fields
with an isolated singularities.}
\end{remark}

\section{Solvability when $L$ satisfies condition $(\mathcal{P})$}
In this section we consider the equation $Lu=f$
 with $f$ a $\cinf$ function
in a neighborhood of 0 under the assumption that $L$ satisfies $\condp$ in
$\R^2\baks\{0\}$. We have the following theorem.

\begin{theorem}
Suppose that the homogeneous vector field $L$ satisfies condition $(\mathcal{P})$ in $\R^2\baks\{0\}$
and $(\mu,\lam)$ is nonresonant. Then, for every $f\in\cinf(\R^2)$ there exist
$\ep>0$ and a continuous function $w(r,\ta)$ in the cylinder $[0,\ \ep)\times\cir$
such that  $w(0,\ta)=0$ and $w\in\cinf((0,\ \ep)\times\cir)$ and such that the function
$u(x,y)$ defined in polar coordinates by $u=w/r^{\lam-1}$ is a distribution
solution of
\begeq
Lu=f\quad\mathrm{in}\quad D(0,\ep)\subset \R^2\, .
\stopeq
\end{theorem}

\begin{proof}
Let $\sum_{j=0}^\infty P_j(x,y)$ be the Taylor series of $f$ at $0\in\R^2$, where
$P_j=r^jf_j(\ta)$ is the homogeneous polynomial given by (6.12). Since the pair $(\mu,\lam)$
is nonresonant, then for every $j\ge 0$, we can find $v_j\in\cinf(\cir)$ such that $r^jv_j(\ta)/r^{\lam-1}$
is a distribution solution of the equation $Lu=P_j$.
By using Borel's Extension Theorem, we can find a function
$v(r,\ta)\in\cinf([0, \ \ep)\times\cir)$, for some $\ep>0$, such that
\begeq
\dd{\,^jv}{r^j}(0,\ta)=j!v_j(\ta)\, ,\quad \forall j\ge 0\, .
\stopeq
Then the $r$-Taylor series of the function
\begeq
g(r,\ta)=f(r,\ta)-L\left(\frac{v(r,\ta)}{r^{\lam-1}}\right)
\stopeq
is identically zero. That is $\dis\dd{\,^jg}{r^j}(0,\ta)=0$ for every $j\ge 0$.
When considered as a function of $(x,y)$ in a neighborhood of $0\in\R^2$, the function
$g$ vanishes to infinite order at $0$.
Since $L$ satisfies $\condp$ in $\R^2\baks\{0\}$, then the vector field
\begeq
L_0=\pt-i\frac{q(\ta)}{p(\ta)}r\pr =\frac{L}{p(\ta)r^{\lam-1}}
\stopeq
satisfies $\condp$ in $\R\times\cir$. Let
\begeq
\Si_0=\{ \ta\in\cir :\ \re{q(\ta)\ov{p(\ta)}}=0\}\, .
\stopeq
Since $\Si_0\subset\cir$ is compact and $L_0$ satisfies $\condp$, then we can find
$\ep>0$ and $\ta_1\,,\cdots ,\ta_N\in\cir$ such that
\begeq\begar{c}
\dis\Si_0\subset \bigcup_{k=1}^N (\ta_k-\ep,\ \ta_k+\ep)\, ,\\
 \dis (\ta_k-2\ep,\ \ta_k+2\ep)\cap(\ta_l-2\ep,\ \ta_l+2\ep) =
\emptyset \ \mathrm{for}\ k\ne l\, .
\stopar\stopeq
Moreover, for each $k=1,\cdots ,N$ there exists
$w_k\in\cinf\left((-\ep,\ \ep)\times (\ta_k-2\ep,\ \ta_k+2\ep)\right)$ such that
\begeq
L_0w_k=\frac{g(r,\ta)}{p(\ta)r^{\lam-1}}\, .
\stopeq
We can furthermore assume that $w_k(0,\ta)=0$. Let $w_0\in\cinf\left((-\ep ,\ \ep)\times\cir\right)$
be such that
\begeq
w_0(r,\ta)=w_k(r,\ta)\ \ \mathrm{in}\
(-\ep,\ \ep)\times (\ta_k-\ep,\ \ta_k+\ep)\, ,\ \ \mathrm{for}\
k=1,\cdots ,N
\stopeq
Hence, $L_0w_0\in \cinf\left((-\ep,\ \ep)\times\cir\right)$ and
\begeq
L_0w_0(r,\ta)=\frac{g(r,\ta)}{p(\ta)r^{\lam-1}}\ \ \mathrm{in}\ \
\bigcup_{k=1}^N(-\ep,\ \ep)\times (\ta_k-\ep,\ \ta_k+\ep)\, .
\stopeq
Consider the equation
\begeq
L_0w^1=h(r,\ta)=\frac{g(r,\ta)}{p(\ta)r^{\lam-1}}-L_0w_0(r,\ta)\,.
\stopeq
Note that $h\equiv 0$ in $\bigcup_{k=1}^N(-\ep,\ \ep)\times (\ta_k-\ep,\ \ta_k+\ep)$.
To solve (7.10), we use the first integral $Z_\mu$ of $L_0$ defined in (2.11).
We have
\begeq
L_0\ov{Z_\mu(r,\ta)}=-\frac{2i}{\mu}\, \frac{\re{q(\ta)\ov{p(\ta)}}}{\abs{p(\ta)}^2}
\ov{Z_\mu(r,\ta)}\, .
\stopeq
It follows from (7.11) that the pushforward via $Z_\mu$ of the equation (7.10) in the region
$0\le r <\ep$ gives rise to the CR equation in a neighborhood of $0\in\C$
\begeq
-\frac{2i}{\mu}\, \frac{\re{q(\ta)\ov{p(\ta)}}}{\abs{p(\ta)}^2}\,
\ov{z}\,\dd{\widehat{w}}{\ov{z}}=\widehat{h}(z)\,,
\stopeq
with $\widehat{w}=w^1\circ Z_\mu^{-1}$ and $\widehat{h}=h\circ Z_\mu^{-1}$. Since
$h$ is identically zero in a an open neighborhood of $(-\ep,\ \ep)\times\Si_0$ and $h(0,\ta)=0$, then
the equation (7.12) can be written in the form
\begeq
\dd{\widehat{w}}{\ov{z}}=\frac{\abs{z}^a}{\abs{z}}\widetilde{h}(z)
\stopeq
with $\widetilde{h}\in\cinf(D(0,\ep')\baks\{0\})\cap L^\infty (D(0,\ep'))$ and
$a=\left(\re{1/\mu}\right)^{-1}>0$. Hence $\abs{z}^{a-1}\widetilde{h}$ is an $L^p$ function
with $p>2$. Equation (7.13) has a solution $\widehat{w}\in C^\sigma(D(0,\ep'))$ with
$\sigma =(p-2)/p$. We can assume that $\widehat{w}(0)=0$. Let
$w^1=\widehat{w}\circ Z_\mu$. The function
\begeq
w(r,\ta)=r^{\lam-1}\left(w^1(r,\ta)+w_0(r,\ta)\right) +v(r,\ta)
\stopeq
is $\cinf$ for $0<r<\ep$,  H\"{o}lder continuous on the circle $r=0$. Also, it follows
from (7.3), (7.9), and (7.10), that $u(r,\ta)=w(r,ta)/r^{\lam-1}$ is a distribution solution
of equation (7.1)
\end{proof}

Equation (7.1) can be solved in the resonant case provided  $f$ satisfies compatibility conditions.

\begin{theorem}
Suppose that $L$ satisfies condition $(\mathcal{P})$ and $(\mu,\lam)$ is resononant.
Let $f$ be a $\cinf$ function
in a neighborhood of $0\in\R^2$ with Taylor series
\begeq
\sum_{m=0}^\infty P_m(x,y)=\sum_{m=0}^\infty r^mf_m(\ta)
\stopeq
where $P_m$ and $f_m$ are given by $(6.12)$. If
\begeq
\int_0^{2\pi}\!\!\frac{f_j(\ta)}{p(\ta)}\exp\left[
-i(j+1-\lam)\psi(\ta)\right]\, d\ta =0\, ,\quad\forall j\in \mathbb{J}(\mu,\lam)
\stopeq
then equation (7.1) has a distribution solution.
\end{theorem}

\section{A boundary value problem for $L$}
In this section, we consider an adaptation of the Riemann-Hilbert boundary value problem
when  the vector field $L$ satisfies $\condp$. Throughout we assume that $\Omega\subset\R^2$
is a simply connected open set containing 0 and having a $C^1$ boundary.
For simplicity, we will assume that
$i^\ast(Bdx-Ady)\ne0$, where $Bdx-Ady$ is the dual form of the vector field $L=A\px+B\py$ and
$i\,: \,\pa\Om\,\longrightarrow\,\R^2$ is the inclusion map.

\begin{theorem}
Suppose that the vector field $L$ with numbers $(\mu,\lam)$ satisfies condition $(\mathcal{P})$,
 $\Lambda\in C^\sigma (\pa\Om,\cir)$ and $\Phi\in C^\sigma (\pa\Om,\R)$ with $0<\sigma <1$.
Denote by $\kappa$  the winding number of $\Lambda$ with respect to $0$. If
\begeq
\kappa > -1-\frac{\mathrm{Re}(\lam)-1}{\mathrm{Re}(1/\mu)}\,
\stopeq
then the Riemann-Hilbert boundary value problem
\begeq\left\{\begar{ll}
Lu=0 &\quad\mathrm{in}\ \ \Om\, ,\\
\mathrm{Re}\left(\Lambda u\right) =\Phi &\quad\mathrm{on}\ \ \pa\Om\, .
\stopar\right.\stopeq
has a solution
\[
u\,\in\, \mcald (\Om)\cap\cinf(\Om\baks\{0\})\, .
\]
\end{theorem}

\begin{proof}
We will use the first integral $Z_\mu$ to convert problem (8.2) into the standard Riemann-Hilbert problem
for the $\ov{\pa}$ in $\C$. We have
\[
L\ov{Z_\mu}=-\frac{2i\mathrm{Re}(p\ov{q})}{\ov{\mu}\,\ov{p}}r^{\lam-1}\ov{Z_\mu}\, ,
\]
$r=\abs{Z_\mu}^{1/\mathrm{Re}(1/\mu)}$, and $L$ is hypocomplex in $\R^2\baks\{0\}$.
It follows that the equation $Lu=0$ in $\Om$ is transformed via
$Z_\mu$ into the equation
\begeq
\abs{z}^{1+\mathrm{Re}(\lam-1)/\mathrm{Re}(1/\mu)}\,\dd{ v}{\ov{z}} =0\, .
\stopeq
with $u=v\circ Z_\mu$. Hence, the boundary value problem (8.2) transforms into
\begeq\left\{\begar{ll}
\dis \abs{z}^{1+\mathrm{Re}(\lam-1)/\mathrm{Re}(1/\mu)}\,\dd{ v}{\ov{z}} =0 &\quad\mathrm{in}\ Z_\mu(\Om)\\
\mathrm{Re}\left(\Lambda_1 v\right) =\Phi_1 &\quad \mathrm{on}\ \pa Z_\mu (\Om)\, ,
\stopar\right.\stopeq
with $\Lambda_1=\Lambda\circ Z_\mu^{-1}$ and $\Phi_1=\Phi\circ Z_\mu^{-1}$.
Note that the presence of the term $\abs{z}^{1+\mathrm{Re}(\lam-1)/\mathrm{Re}(1/\mu)}$ allows us to
seek solutions of the form
\[
v(z)=\frac{H(z)}{z^m}
\]
with $H$ holomorphic and $m\in\Z^+$ satisfying
\[
m<1+\frac{\mathrm{Re}(\lam-1)}{\mathrm{Re}(1/\mu)}\, .
\]
We can get explicit solutions by reducing the problem to the unit disc. Let
\[
\vartheta\, :\, Z_\mu(\Om)\,\longrightarrow\, D(0,1)
\]
be a conformal mapping with $\vartheta (0)=0$. The boundary problem for $w=v\circ\vartheta^{-1}$ is
therefore
\begeq\left\{\begar{ll}
\dis \abs{z}^{1+\mathrm{Re}(\lam-1)/\mathrm{Re}(1/\mu)}\,\dd{ w}{\ov{z}} =0 &\quad\mathrm{in}\ D(0,1)\\
\mathrm{Re}\left(\Lambda_2 w\right) =\Phi_2 &\quad \mathrm{on}\ \pa D(0,1)\, ,
\stopar\right.\stopeq
with $\Lambda_2=\Lambda_1\circ \vartheta^{-1}$ and $\Phi_2=\Phi_1\circ \vartheta^{-1}$.
Since on the boundary $\pa\Om$ we have $i^\ast dZ_\mu\ne 0$, then $Z_\mu$ is a $C^1$-diffeomorphism from
$\pa\Om$ onto $\pa Z_\mu(\Om)$ and, consequently, $\vartheta\circ Z_\mu$ is a diffeomorphism from $\pa\Om$ onto
the unit circle $\pa D(0,1)$. Hence, the functions $\Lambda_2$ and $\Phi_2$ are H\"{o}lder continuous on the unit
circle.

The explicit solution of (8.5) can be obtained through the Schwarz operator as follows (see {\cite{Beg}},
{\cite{Gak}}, {\cite{Mus}}). Let
\begeq\begar{l}
\dis \gamma(z)=
\frac{1}{2\pi}\int_0^{2\pi}\!\!
\left[ \arctan\left(\frac{\mathrm{Im}(\Lambda_2(\tau))}{\mathrm{Re}(\Lambda_2(\tau))}\right)-\kappa\tau\right]
\,\frac{\ei{i\tau}+z}{\ei{i\tau}-z}\, d\tau\\
\gamma(z)=\gamma_1(z)+i\gamma_2(z)\\
\dis Q(z)=i\beta_0+\sum_{k=1}^n\left(c_kz^k-\ov{c_k}\,z^{-k}\right)\,\quad \beta_0\in\R,\ \ c_1,\cdots ,c_n\,\in \C\\
n=\kappa +1 +\left[\frac{\mathrm{Re}(\lam)-1}{\mathrm{Re}(1/\mu)}\right]\, .
\stopar\stopeq
The general solution of (8.5) is the meromorphic function with a possible pole at 0 given by
\begeq
w(z)=z^\kappa \ei{i\gamma(z)}\left[ \frac{1}{2\pi}\int_0^{2\pi}\!\!\ei{i\gamma_2(\tau)}
\Phi_2(\tau)\frac{\ei{i\tau}+z}{\ei{i\tau}-z}\, d\tau\, +\, Q(z)
\right]\stopeq
The distribution
$\dis u=w\circ\vartheta\circ Z_m \in \mcald (\Om)\cap\cinf(\Om\baks\{0\})$ solves (8.2)
\end{proof}

\begin{theorem}
Assume that the vector field $L$ satisfies condition $(\mathcal{P})$ and that the associated
pair $(\mu,\lam)$ is nonresonant and satisfies the diophantine condition $\mcaldc$. Then, for
every $f\in\cinf(\ov{\Om})$ there exists $u\in\mcald (\Om)\cap\cinf(\ov{\Om}\baks\{0\})$ such
that
\begeq
Lu=f\, .
\stopeq
\end{theorem}

\begin{proof}
Since $L$ satisfies $\condp$ and $(\mu,\lam)$ is non resonant and satisfies $\mcaldc$, then
(Theorem 7.1) there exists $u\in\mcald (U_0)\cap\cinf(U_0\baks\{0\})$, with
$U_0=D(0,\ep)$, for some $\ep>0$, such that $u$ satisfies (8.8) in $U_0$.
It follows from $\condp$ that we can find open sets $U_1,\cdots ,U_N$ such that
\[
\ov{\Om}\baks D(0,\ep)\,\subset\, \bigcup_{j=1}^NU_j
\]
and functions $u_j\in\cinf(U_j)$ satisfying (8.8) in $U_j$ for $j=1,\cdots ,N$.
If $j,k\in\{0,\cdots ,N\}$ are such that $U_j\cap U_k\ne \emptyset$, then $L(u_j-u_k)=0$
in $U_j\cap U_k$ and, consequently, there exist $h_{jk}$ holomorphic in
$Z_\mu(U_j\cap U_k)=Z_\mu(U_j)\cap Z_\mu(U_k)$ such that $u_j-u_k=h_{jk}\circ Z_\mu$
(hypocomplexity of $L$). The collection $\{ h_{jk}\}$ forms a cocycle relative to the covering
$\{Z_\mu(U_j)\}_{j=0}^N$. Hence, we can find holomorphic functions $h_j$ in $Z_\mu(U_j)$ such
that $h_{jk}=h_j-h_k$ in $Z_\mu(U_j)\cap Z_\mu(U_k)$. The distribution $u\in\mcald(\Om)$
given by
\[
u=u_j-h_j\circ Z_\mu\quad\mathrm{in}\ \ U_j
\]
solves (8.8)
\end{proof}

\begin{theorem}
Let $L$ and $(\mu,\lam)$ be as in Theorem $(8.2)$ and $\Lambda$, $\Phi$ be as in Theorem $(8.1)$
with $\kappa =\mathrm{Ind}(\Lambda)$ satisfying $(8.1)$. Then for any $f\in\cinf(\ov{\Om})$,
the Riemann-Hilbert problem
\begeq\left\{\begar{ll}
Lu =f &\quad \mathrm{in}\ \ \Om\\
\mathrm{Re}\left(\Lambda u\right) =\Phi &\quad\mathrm{on}\ \ \pa\Om
\stopar\right.\stopeq
has a solution $u\in\mcald(\Om)\cap\cinf(\Om\baks\{0\})$.
\end{theorem}

\begin{proof}
Let $v\in\mcald(\Om)\cap\cinf(\ov{\Om}\baks\{0\})$ be such that $Lv=f$ (Theorem 8.2).
Let $w$ be the solution of the problem (Theorem 8.1)
\[
Lw=0\ \ \mathrm{in}\ \Om\, ,\quad \mathrm{Re}(\Lambda w)=\Phi -\mathrm{Re}(\Lambda v)\ \
\mathrm{on}\ \pa\Om\, .
\]
Then $u=w+v$ solves (8.9)
\end{proof}

\bibliographystyle{amsplain}

\end{document}